\documentclass[11pt,a4paper,reqno]{amsart}

\usepackage{amsmath, amsthm, amsfonts}
\usepackage[utf8]{inputenc}
\usepackage[colorlinks, linkcolor=blue, citecolor=blue, urlcolor=blue]{hyperref}
\newcommand{\abs}[1]{\left\vert#1\right\vert}

\newcommand{\norm}[1]{\left\Vert#1\right\Vert}

\def\tnorm{\vert\!\vert\!\vert}

\def\R{\mathbb{R}}

\DeclareMathOperator{\Lip}{Lip}

\newtheorem{thm}{Theorem}[section]

\newtheorem{lem}[thm]{Lemma}

\theoremstyle{definition}
\newtheorem{dfn}[thm]{Definition}
\theoremstyle{remark}
\newtheorem{rem}[thm]{Remark}
\renewcommand{\theenumi}{\roman{enumi}}

\begin{document}

\title[Measure solutions for some models in population dynamics]{Measure solutions for some models in population dynamics}

\author[J. A. Ca\~nizo]{Jos\'e A. Ca\~nizo}
\address{Departament de Matem\`atiques, Universitat Aut\`onoma de Barcelona,
  08193 Bellaterra (Barcelona), Spain.}
\email{canizo@mat.uab.es}

\author[J. A. Carrillo]{Jos\'e A. Carrillo}
\address{ICREA and Departament de Matem\`atiques, Universitat Aut\`onoma de Bar\-ce\-lo\-na,
08193 Bellaterra (Barcelona), Spain. \textit{On leave from:}
Department of Mathematics, Imperial College London, London SW7
2AZ, UK.}
\email{carrillo@mat.uab.es}

\author[S. Cuadrado]{S\'\i lvia Cuadrado} \address{Departament de
  Matem\`atiques, Universitat Aut\`onoma de Barcelona, 08193
  Bellaterra (Barcelona), Spain.}
\email{silvia@mat.uab.cat}

\thanks{The authors were partially supported by the Ministerio de Ciencia e Innovaci\'on, grant
MTM2011-27739-C04-02, and by the Ag\`encia de Gesti\'o d'Ajuts
Universitaris i de Recerca-Generalitat de Catalunya, grant
2009-SGR-345.}

\keywords{population dynamics, selection-mutation, measure solutions,
  transport distances}

\subjclass[2010]{35Q92, 92D15, 92D25}

\begin{abstract}
  We give a direct proof of well-posedness of solutions to general
  selection-mutation and structured population models with measures as
  initial data. This is motivated by the fact that some stationary
  states of these models are measures and not $L^1$ functions, so the
  measures are a more natural space to study their dynamics. Our
  techniques are based on distances between measures appearing in
  optimal transport and common arguments involving Picard
  iterations. These tools provide a simplification of previous
  approaches and are applicable or adaptable to a wide variety of
  models in population dynamics.
\end{abstract}

\maketitle

\section{Introduction}

Selection-mutation equations are models for structured populations
with respect to continuous phenotypical evolutionary traits. They are
usually written as equations for densities on the parameter space of
phenotypes \cite{B,cc,magalwebb,ccdr}, that is, they are usually
formulated in $L^{1}$ spaces. However, for some of these models a more
natural space to study the time evolution is the space of positive
measures, since it has been proven in some cases \cite{cc,c} that for
a small mutation rate the steady states tend to concentrate in a Dirac
mass at the evolutionarily stable strategy value. More generally, the
need for a theory of well-posedness in measures for structured
population models was mentioned in
\cite{DiekmannMetz,DiekmannEtAl,WebbStructured}.

Some efforts in this direction have been directed at particular models
in population dynamics: pure selection models for phenotypic traits in
the space of measures were recently studied in
\cite{acklehfitzthieme,cressman}, while in \cite{burger} a particular
case of a selection-mutation equation for a genetic trait in the space
of measures is analyzed. More recently in \cite{CA} the author shows
well-posedness and studies asymptotic behavior of a selection-mutation
equation.

Our aim here is to give a simple and general proof of well-posedness
in the space of measures for a class of models that will include a
wide range of selection-mutation models as well as many classical
nonlinear structured population models \cite{B, DiekmannEtAl,
  DiekmannMetz, Perthame, WebbStructured, WebbTheoryNonLinear}. The
basic model we consider is an abstract Cauchy problem of the form
\begin{subequations}
  \label{eq:abstract}
  \begin{gather}
    \label{eq:abstract-eq}
    \partial_t u + \nabla \cdot(F(x) u) = N(t,u),
    \\
    \label{eq:abstract-in}
    u(0,x) = u_0(x) \qquad (x \in \R^d).
  \end{gather}
\end{subequations}
where $u = u(t,x)$ is the unknown, which depends on $t \geq 0$ and $x
\in \R ^d$ (in any dimension $d \geq 1$). Any differential terms in
the equation should be included in the term $\nabla \cdot(F(x) u)$
(e.g., growth or aging terms); the term $N(t,u)$ (a measure depending
on $x$) may include nonlinear birth and death rates and
selection-mutation interactions. Several examples in population
dynamics are given in Section \ref{sec:examples}, where it is also
shown how the abstract result may be applied to equations whose domain
is not the whole space $\R^d$.

Our proof of well-posedness in the space of measures to
\eqref{eq:abstract} will be based on techniques stemming from mass
transportation and semigroup theory. Similar ideas were already used
for studying the mean-field limit of kinetic equations such as the
classical Vlasov equation \cite{D} and more recently in swarming
models \cite{CCR}. The main advantage of our approach is its
simplicity, coming from the use of techniques already well developed
in other fields, and its flexibility, which allows it to be adapted to
a wide range of models in population dynamics.

The well-posedness of measure solutions to some equations of the form
\eqref{eq:abstract} has recently been analysed by different although
related techniques in \cite{GM,GwiazdaThomasEtAl}.  selection-mutation
models were not included in their formulation but
Sharpe-Lotka-McKendrick-type models (age-structured models) in which
the boundary condition is introduced as a measure-valued right-hand
side of the equation are treated, see subsection 3.3 for related
results.

A nonlinear semigroup approach using the splitting method for the
transport $\nabla \cdot(F(x) u)$ and the right-hand side $N(t,u)$
terms was introduced in the more recent paper \cite{CCGU} to treat
equations of the form \eqref{eq:abstract}. Here, we give a direct and
simpler proof based on Picard iterations in the right metric space to
conclude existence, uniqueness, and continuous dependence without
resorting to the splitting method or nonlinear semigroup techniques.

The organization of the paper is as follows: in Section 2 we show that
the abstract Cauchy problem \eqref{eq:abstract} is well posed for
measures as initial data in the so-called bounded Lipschitz distance
under reasonable Lipschitz conditions on $F$ and $N$ similar to the
ones needed in \cite{CCGU}. Then, Section 3 is devoted to the
application of these results to more explicit examples, mainly some
nonlinear selection-mutation models in subsections 3.1 and 3.2 where
$F=0$, but also some mixed nonlinear structured
population/selection-mutation models in subsection 3.3 and pure
structured population models in subsection 3.4 where $F\neq 0$. These
applications highlight the wide applicability of the abstract theorem
in Section 2 for this type of models, setting a possible functional
framework for stability and asymptotic convergence towards measure
solutions \cite{acklehfitzthieme,cc}.

%%%%%%%%%%%%%%%%%%%%%%%%%%%%%%%%%%%%%%%%%%%%%%%%%%%%%%%%%%%%%%%%%%%%%

\section{Well-posedness theory}

\subsection{The bounded Lipschitz norm}

Let us start with a quick summary of the definition and properties of
the bounded Lipschitz norm, also called flat metric
\cite{Spohn,Villani}. We denote by $\mathcal{M}(\R^d)$ the
set of Radon measures in $\R^d$, and consider the space of Lipschitz
functions $W^{1,\infty}(\R^d)$ endowed with the norm
$\norm{\psi}_{1,\infty}:= \norm{\psi}_\infty + \Lip(\psi)$, with
$\Lip(\psi)$ the Lipschitz constant of $\psi$.

\begin{dfn}\label{dfn:norm}
Given a Radon measure $\mu \in \mathcal{M}(\R^d)$ we define its
\emph{bounded Lipschitz norm} $\norm{\mu}_{\mathcal{M}(\R^d)}$,
$\norm{\mu}$ when there is no ambiguity, by
  \begin{equation*}
%    \label{eq:bln}
    \norm{\mu} :=
    \sup_{\psi \in \mathcal{L}} \abs{
      \int_{\R^d} \psi\, d\mu
    },
  \end{equation*}
  where $\mathcal{L}$ is the set of bounded functions $\psi:\R^d \to
  \R$ with $\Lip(\psi)\leq 1$ and $\|\psi\|_\infty \leq 1$, i.e.,
  $\mathcal{L} := \{ \psi \in W^{1,\infty}(\R^d) \mid
  \norm{\psi}_{1,\infty} \leq 1 \}$.
\end{dfn}
One sees from the definition that this is just the dual norm of
$W^{1,\infty}(\R^d)$, and that by duality for any $\psi \in
W^{1,\infty}(\R^d)$,
\begin{equation*}
  \int_{\R^d} \psi \, d\mu \leq \|\psi\|_{1,\infty} \|\mu\|.
\end{equation*}
We remark that on probability measures this can also be defined as a
distance of Kantorovich-Rubinstein, or Wasserstein, type: when $\mu$,
$\nu$ are probability measures it holds that
\begin{equation*}
%  \label{eq:wd}
  \norm{\mu - \nu}
  =
  \inf_{\sigma \in \Sigma} \,
  \int_{\R^d \times \R^d}
  \min\{\abs{x-y}, 1\}\, d\sigma(x,y) \,,
\end{equation*}
where $\Sigma$ is the set of \emph{transference plans} between $\mu$
and $\nu$, that is, probability measures on $\R^d \times \R^d$ with
marginals $\mu$ and $\nu$ \cite{Villani}.

We will need to use the following simple result on this distance:
\begin{lem}
  \label{lem:product-bound}
  If $b \in W^{1,\infty}(\R^d)$ and $\mu \in \mathcal{M}(\R^d)$, then
  $b \mu \in \mathcal{M}(\R^d)$ and
  \begin{equation*}
%    \label{eq:product-bound}
    \| b \mu \| \leq \|b\|_{1,\infty} \|\mu\|.
  \end{equation*}
\end{lem}

\begin{proof}
  It is clear that $b \mu \in \mathcal{M}(\R^d)$, since $b$ is a
  bounded continuous function. We integrate against $\psi \in
  W^{1,\infty}$ with $\|\psi\|_{1,\infty} \leq 1$ to find
  \begin{equation}
    \label{l1}
    \left| \int_{\R^d} b \psi \,d\mu \right|
    \leq
    \| b \psi \|_{1,\infty} \|\mu\|
  \end{equation}
  On the other hand, we have
  \begin{gather*}
    \|b \psi\|_{\infty} \leq \|b\|_\infty \|\psi\|_\infty,
    \\
    \Lip(b \psi)
    \leq \Lip(b) \|\psi\|_\infty
    + \|b\|_\infty \Lip(\psi) ,
  \end{gather*}
  so
  \begin{align*}
    \|b \psi\|_{1,\infty}
    &=
    \|b \psi\|_{\infty} + \Lip(b \psi)
    \leq
    \|b\|_\infty \big(\|\psi\|_\infty + \Lip(\psi)\big)
    + \Lip(b) \|\psi\|_\infty\\
    &\leq
    \|b\|_\infty + \Lip(b) = \|b\|_{1,\infty}.
  \end{align*}
  Plugging this into eq. (\ref{l1}) finishes the proof.
\end{proof}

In the rest of this paper, we will work with measure solutions to some
evolution partial differential equations and therefore, we will work
with the space of bounded continuous curves on the set of measures
$BC(I;\mathcal{M}(\R^d))$ depending on $t \in I$ denoting the time
variable, with $I=[0,T]$ for some $T>0$ or $I=[0,\infty)$. The
continuity of the curves of measures $t \mapsto \mu(t)$ is always
understood to be with respect to the bounded Lipschitz norm. We warn
the reader that elements in $BC(I;\mathcal{M}(\R^d))$ will often be
denoted as if they were absolutely continuous densities with respect
to Lebesgue with the form $d\mu(t) = u(t,x)\,dx$ for the sake of
simplicity.

The standard total variation norm for measures will be denoted by
\mbox{$\|\cdot\|_{\mathrm{TV}}$}. We remark the natural necessity of
the bounded Lipschitz distance (or similar distances between measures)
to work with transport evolution equations, as opposed to the total
variation norm. In fact, take any injective continuous path
$x:[0,T]\longrightarrow \R^d$ and take the curve of measures $\mu$
defined by $t\longrightarrow \delta_{x(t)}$. It is easy to check that
$\mu$ belongs to $BC([0,T];\mathcal{M}(\R^d))$ while
$\|\mu(t)-\mu(s)\|_{\mathrm{TV}}=2$ for all $0\leq t < s\leq T$.

Although all models in population dynamics study the evolution of
positive measures (number density of individuals with respect to
some variables), let us mention that we need to use the bounded
Lipschitz norm and not other optimal transport distances since the
total mass (total variation) of measure solutions will typically
not be preserved in time. We will denote by $B_{BL}(R)$, resp.
$B_{\mathrm{TV}}(R)$, the ball of radius $R$ centered at $0$ in the Bounded
Lipschitz norm and in the total variation norm resp. Finally, let us
point out that balls $B_{\mathrm{TV}}(R)$ in $\mathcal{M}(\R^d)$ with
respect to the total variation norm are closed in the bounded
Lipschitz norm by simple weak convergence arguments.

\subsection{An abstract result}

We consider the abstract evolution equation for measures given in
\eqref{eq:abstract}, which we recall here:
\begin{align*}
    &\partial_t u + \nabla_x \cdot(F(x) u) = N(t,u),
    \\
    &u(0,x) = u_0(x) \qquad (x \in \R^d).
\end{align*}
Here $u = u(t,x)$ is the unknown, which depends on $t \geq 0$ and
$x \in \R ^d$, under the following hypotheses on $F$, $N$ and
$u_0$:
\renewcommand{\theenumi}{H\arabic{enumi}}
\begin{enumerate}
\item \label{H1}
  $u_0 \in \mathcal{M}(\R^d)$.
\item \label{H2}
  $F: \R^d \to \R^d$ is a bounded Lipschitz map.
\item \label{H3} $N: [0,+\infty) \times \mathcal{M}(\R^d) \to
  \mathcal{M}(\R^d)$ is a continuous function both in $t$ and $u$
  (considering here the topologies induced by the bounded Lipschitz
  norm on $\mathcal{M}(\R^d)$ and the usual topology on
  $[0,+\infty)$.)
\item \label{H4}
  $N$ is locally Lipschitz in its second variable, i.e., for
  every bounded set $K \subseteq [0,+\infty) \times \mathcal{M}(\R^d)$
  there exists $L_N = L_N(K) > 0$ such that
  \begin{equation*}
%    \label{eq:N-loc-Lipschitz}
    \| N(t,u_1) - N(t,u_2) \| \leq L_N \|u_1 - u_2\|
    \quad \forall\ (t,u_1), (t,u_2) \in K.
  \end{equation*}
\item \label{H5}
  $N$ carries bounded sets in the total variation norm to
  bounded sets in the total variation norm: for each $R
  > 0$ there exists $C_R \geq 0$ such that
  $\| N(t,u) \|_{\mathrm{TV}} \leq C_R$
  for all $t \geq 0$ and $u \in \mathcal{M}(\R^d)$ with $\| u \|_{\mathrm{TV}}
  \leq R$.
\end{enumerate}
\renewcommand{\theenumi}{\roman{enumi}}

\begin{dfn}
  \label{dfn:solution-abstract}
  Assume Hypotheses \eqref{H1}-\eqref{H5}, and take $T \in
  (0,+\infty]$. We say $u\in BC([0,T);\mathcal{M}(\R^d))$
  is a \emph{solution of equation \eqref{eq:abstract} on $[0,T)$ with
  initial condition $u_0$} when, for every $\phi \in \mathcal{C}^\infty_0([0,T) \times \R^d)$,
  \begin{multline*}
%    \label{eq:def-solution}
      - \int_0^T \int_{\R^d} \partial_t \phi(t,x) \,u(t,x) \,dx -
      \int_{\R^d} \phi(0,x) \,u_0(x) \,dx
      \\
      - \int_0^T \int_{\R^d} F(x)
      \nabla \phi(t,x) \,u(t,x)\,dx \,dt
      = \int_0^T \int_{\R^d} \phi(t,x) N(t,u)(x) \,dx \,dt
      .
  \end{multline*}
\end{dfn}

\begin{thm}[Well-posedness of the abstract equation]
  \label{thm:well-posedness}
  Assume Hypotheses \eqref{H1}-\eqref{H5}. There exists a maximal time $T
  > 0$ such that there is a unique solution $u\in B([0,T);\mathcal{M}(\R^d))$
  of equation \eqref{eq:abstract}. In addition:
  \begin{enumerate}
  \item Either $T = +\infty$ or $\lim_{t \to T}\|u(t)\|_{\mathrm{TV}} =
    +\infty$.
  \item This solution depends continuously on the initial condition
    $u^0$ in the bounded Lipschitz norm: take two solutions
    $u_1,u_2$ of equation \eqref{eq:abstract} on $[0,T)$ with initial
    conditions $u_1^0$, $u_2^0$, respectively. Assume also that
  \begin{equation*}
%    \label{eq:u-bounded}
    \|u_1(t)\|, \|u_2(t)\| \leq R
    \qquad (t \in [0,T)),
  \end{equation*}
  and take $L_N$ to be the Lipschitz constant of $N$ with respect to
  the second variable on the set $[0,T) \times B_{BL}(R) \subseteq
  [0,+\infty) \times \mathcal{M}(\R^d)$. Then,
  \begin{equation*}
%    \label{eq:continuous-dependence}
    \|u_1(t) - u_2(t)\|
    \leq e^{(L_F+L_N) t} \|u_1^0 - u_2^0\|
    \qquad (t \in [0,T)).
  \end{equation*}
  \end{enumerate}
\end{thm}

\begin{rem}
  \label{rem:in-omega}
  If there is no drift term present in equation \eqref{eq:abstract}
  (this is, $F = 0$) and $\Omega \subseteq \R^d$ is an open set, then
  the result holds if one changes $\mathcal{M}(\R^d)$ by
  $\mathcal{M}(\Omega)$. The modifications needed in the proof below
  are straightforward and we omit them.
\end{rem}

\begin{proof}[Proof of Theorem {\rm\ref{thm:well-posedness}}]
  Define $X_t: \R^d \to \R^d$ as the flow at time $t$ of the
  characteristic equations
  \begin{equation*}
%    \label{eq:charac}
    \frac{dX}{dt} = F(X).
  \end{equation*}
  By standard arguments in the theory of ordinary differential
  equations we have that $L_{X_t}$, the Lipschitz constant of the flow
  $X_t$ at time $t$, satisfies
  \begin{equation*}
    L_{X_t} \leq e^{L_F t} \qquad (t \geq 0),
  \end{equation*}
  where $L_F$ is the Lipschitz constant of $F$.

  We will prove existence by a fixed point argument in the set
  $$
  \mathcal{M}_{T} := \{ u \in BC([0,T]; B_{\mathrm{TV}}(R)) \mid u(0) = u^0 \},
  $$
  of radius $R := 2 \|u_0\|_{\mathrm{TV}}$. We choose
  \begin{equation}
    \label{eq:def-T}
    T := \min \{ \|u^0\|_{\mathrm{TV}}/C_R, 1 / L_F, 1/(3L_N) \},
  \end{equation}
  where $C_R$ and $L_N$ are given in Hypotheses (H1)-(H5) with
  $K=[0,T) \times B_{BL}(R)$. We endow
  $\mathcal{M}_T$ with the standard norm
  \begin{equation*}
    %\label{eq:M_T-norm}
    \tnorm u \tnorm := \sup_{t \in [0,T]} \| u(t) \|\,.
  \end{equation*}
  As noticed before $\mathcal{M}_T$ is a closed subset of the
  space $BC([0,T]; B_{\mathrm{TV}}(R))$. Hence $\mathcal{M}_T$ is complete
  and we may apply the Banach fixed point theorem in it.
  Define the map
  $\Gamma: \mathcal{M}_T \to \mathcal{M}_T$ as
  \begin{equation}
    \label{eq:Gamma}
    \Gamma(u)(t)
    :=
    X_t \# u^0 + \int_0^t X_{t-s} \# N(s,u(s)) \,ds\,,
  \end{equation}
  with $\#$ denoting the push-forward of a measure through a map. It
  is easy to check that a fixed point of this map is in fact a
  solution to equation \eqref{eq:abstract}.

{\bf Step 1: $\Gamma$ is well-defined.-}
  Let us first show that $t \mapsto X_t \# u^0$ and $s \mapsto X_{t-s}
  \# N(s,u(s))$ (for $t \in [0,T]$ fixed) are continuous maps. For the first
  one, taking any test function $\phi \in \mathcal{L}$ and any $t,
  \tau \in [0,T]$,
  \begin{align*}
    \int_{\R^d} \phi \cdot \big(X_t \# u^0 - X_\tau \# u^0\big)\,dx
    &=
    \int_{\R^d} (\phi(X_t(x)) - \phi(X_\tau(x)))u^0(x) \,dx
    \\
    &\leq
    \int_{\R^d} |X_t(x) - X_\tau(x)|\, |u^0|(x) \,dx
    \\
    &\leq
    |t-\tau| \,\|F\|_\infty \|u^0\|_{\mathrm{TV}},
  \end{align*}
  which shows continuity. As for the second one, we take $\phi$ as
  before, fix $t \in [0,T]$ and take any $\tau,s \in [0,t]$. Denoting
  $N(s,u(s))$ as $N_s$ for short, we have
  \begin{multline*}
    \int_{\R^d} \!\phi \cdot \big( X_{t-s} \# N_s
    - X_{t-\tau} \# N_\tau \big)\,dx
    \\
    =
    \int_{\R^d} \phi \cdot \big( X_{t-s} \# N_s
    - X_{t-s} \# N_\tau \big)\,dx
    +
    \int_{\R^d} \phi \cdot \big( X_{t-s} \# N_\tau
    - X_{t-\tau} \# N_\tau \big)\,dx
    \\
    =
    \int_{\R^d} (\phi \circ X_{t-s}) \, (N_s - N_\tau)\,dx
    +
    \int_{\R^d} \!\!\big( \phi(X_{t-s}(x)) - \phi(X_{t-\tau}(x)) \big)
    \,N_\tau(x) \,dx
    \\
    \leq
    L_{X_{t-s}} \|N_s - N_\tau\|
    + |\tau - s| \,\|F\|_{\infty} \|N_\tau(x)\|_{\mathrm{TV}}
    \\
    \leq
    e^{(t-s) L_F} \|N_s - N_\tau\|
    + C_R \|F\|_{\infty} |\tau - s|\,.
  \end{multline*}
  This proves continuity, as $s \mapsto N(s,u(s))$ is continuous due
  to (H3). Hence,
  the integral in \eqref{eq:Gamma} makes sense, $\Gamma(u)$ is continuous
  from $[0,T]$ to $\mathcal{M}(\R^d)$ in the bounded Lipschitz norm,
  and we only need to see that its image is inside $B_{\mathrm{TV}}(R)$:
  \begin{align*}
    \|\Gamma(u)(t)\|_{\mathrm{TV}} &\leq
    \norm{X_t \# u^0}_{\mathrm{TV}}
    +
    \int_0^t \norm{ X_{t-s} \# N(u(s)) }_{\mathrm{TV}} \,ds
    \\
    &\leq
    \norm{u^0}_{\mathrm{TV}}
    +
    \int_0^t \norm{N(u(s)) }_{\mathrm{TV}} \,ds
    \\
    &\leq
    \norm{u^0}_{\mathrm{TV}}
    +
    C_R T
    \leq
    2 \norm{u^0}_{\mathrm{TV}} = R\,.
  \end{align*}

{\bf Step 2: $\Gamma$ is contractive.-}
  Take $u,v \in \mathcal{M}_T$. Using similar arguments we estimate
  \begin{align*}
    \|\Gamma(u)(t) - \Gamma(v)(t)\|
    &\leq
    \int_0^t \| X_{t-s} \# N(s,u(s)) - X_{t-s} \# N(s,v(s)) \| \,ds
    \\
    &\leq
    \int_0^t L_{X_{t-s}} \| N(s,u(s)) - N(s,v(s)) \| \,ds
    \\
    &\leq
    e^{L_F T} L_N
    \int_0^t  \| u(s) - v(s) \| \,ds.
  \end{align*}
  By taking the maximum over $t \in [0,T]$ this implies
  \begin{equation*}
    \tnorm \Gamma(u) - \Gamma(v) \tnorm
    \leq
    e^{L_F T} L_N T \, \tnorm u - v \tnorm
    < L \, \tnorm u - v \tnorm,
  \end{equation*}
  for some $L < 1$, due to the choice of $T$ made in
  \eqref{eq:def-T}. An application of the Banach fixed point theorem
  together with usual arguments on the extension of solutions finishes
  the proof of point i) of the theorem.

{\bf Step 3: Continuous dependence.-}
  We estimate the difference of the two solutions as follows:
  \begin{align*}
    \|u(t) &- v(t)\|
    \\
    &\leq
    \|X_t \# u^0 - X_t \# v^0\|
    + \int_0^t
    \| X_{t-s} \# N(s,u(s)) - X_{t-s} \# N(s,v(s)) \| \,ds
    \\
    &\leq
    L_{X_{t}} \|u^0 - v^0\|
    + \int_0^t
    L_{X_{t-s}} \| N(s,u(s)) - N(s,v(s)) \| \,ds
    \\
    &\leq
    e^{L_F t} \|u^0 - v^0\|
    + L_N \int_0^t
    e^{L_F (t-s)} \| u(s) - v(s) \| \,ds.
  \end{align*}
  Gronwall's Lemma then implies the result.
\end{proof}

\begin{rem}
Theorem \ref{thm:well-posedness} is a generalization of ideas in
the theory of linear evolution semigroups \cite{Pazy,Nagel}, since
equation \eqref{eq:abstract} is the sum of a linear term, and a
locally Lipschitz perturbation. However, a small modification of
the argument is needed: this comes from the fact that one cannot
work in the dual space $[W^{1,\infty}(\R^d)]^*$. Actually, the
proof above shows that by restricting to measures in $B_{\mathrm{TV}}(R)$
we are able to prove the continuity of the transport semigroup.
This continuity is not evident in $B_{BL}(R)$.
\end{rem}

Finally, we point out that we are usually interested in positive
measures as initial condition, even if Theorem
\ref{thm:well-posedness} does not require positivity. It is most often
the case in models that the quantity under study is the density of a
given population, which is intrinsically positive. Hence, for most
models of interest, positivity is preserved in time (see for example
Lemma \ref{lem:positivity}).

%%%%%%%%%%%%%%%%%%%%%%%%%%%%%%%%%%%%%%%%%%%%%%%%%%%%%%%%%%%%%%%%%%%%%%%%%%%

\section{Application to particular models}
\label{sec:examples}

In this section we will apply Theorem \ref{thm:well-posedness} to show
well-posedness of four particular models in population dynamics. The
first one is a simple selection-mutation equation for a phenotypic
variable inspired by the ``continuum of alleles model'' introduced by
Crow and Kimura (see \cite{crow} and also \cite{burger}) in the field
of population genetics in order to explain the maintenance of genetic
variation due to the balance effect of selection and mutation. The
second one, introduced in \cite{ccdr}, is a modification of the first
one in which it is assumed that the nonlinear term modelling the
competition between individuals for resources is
infinite-dimensional. The third model we consider was studied in
\cite{palmada} and it is a selection-mutation model for an
age-structured population. The last model we present is an age- and
size-structured model that was introduced in \cite{W}.

\subsection{A simple selection-mutation equation}
\label{sec:selection-mutation}

Let us consider the following selection-mutation equation:
\begin{multline}
  \label{eq:spe}
  \frac{\partial u}{\partial t} (t,x)
  = (1-\varepsilon) b(x) u(t,x) - m(x,P(t)) u(t,x)
  \\
  +\varepsilon \int_{\Omega} b(y) \gamma(x,y) u(t,y) \,dy
  := N(u(t,\cdot))(x)
\end{multline}
for the density $u(t,x)$ of individuals at time $t \geq 0$ with
respect to an (abstract) evolutionary variable $x$ in an open set
$\Omega \subseteq \R^d$. $P(t)$ denotes the total population at time
$t$
\begin{equation*}
%  \label{eq:P}
  P(t) := \int_{\Omega} u(t,x) \,dx
\end{equation*}
and $m$ is the trait-specific death rate which depends in an
increasing way on the total population $P(t)$ at time $t$. The
inflow of non-mutant newborns will be given by $(1-\varepsilon)
b(x) u(t,x)$ where $b(x)$ is the trait-specific fertility and
$\varepsilon$ stands for the probability of mutation. The inflow
of mutant newborns will be given by the integral operator
$\varepsilon \int_{\Omega} b(y) \gamma(x,y) u(t,y) \,dy$ where
$\gamma(x,y)$ is the density of probability that the trait of the
mutant offspring of an individual with trait $y$ is $x$.

We may apply Theorem \ref{thm:well-posedness} to equation
\eqref{eq:spe} under the following conditions:

\begin{thm}
  \label{thm:selection-mutation}
  Assume \eqref{H1} and also that $b$, $m$ and $\gamma$ satisfy the
  following:
  \begin{enumerate}
  \item \label{it:sm-b} $b: \Omega \to \R$ is in $W^{1,\infty}$ (i.e.,
    it is bounded and Lipschitz).
  \item \label{it:sm-m}
    $m: \Omega \times \R \to \R$ satisfies that for each $p \in
    \R$, $m(\cdot, p) \in W^{1,\infty}$, and for each $R > 0$ there
    exists $L_m > 0$ such that
    \begin{equation}
      \label{eq:hyp-m2}
      \norm{m(\cdot, p_1) - m(\cdot, p_2)}_{1,\infty}
      \leq
      L_m |p_1 - p_2|
      \quad \text{ for all } p_1, p_2 \in [-R,R].
    \end{equation}
  \item \label{it:sm-gamma}
    For each $y \in \Omega$, $\gamma(\cdot, y)$ is a positive
    probability measure on $\Omega$ and
    there exists $L_\gamma > 0$ such that
    \begin{equation}
      \label{eq:hyp-gamma-Lip}
      \|\gamma(\cdot,y) - \gamma(\cdot,z)\|
      \leq
      L_\gamma |y-z|
      \quad \text{ for all } y,z \in \Omega.
    \end{equation}
  \end{enumerate}
  Then the operator $N$ in equation \eqref{eq:spe} satisfies the
  hypotheses of Theorem {\rm\ref{thm:well-posedness}}. Consequently,
  equation \eqref{eq:spe} is well-posed in the sense of Theorem
  {\rm\ref{thm:well-posedness}}.
\end{thm}

\begin{proof}
  We need to check that assumptions \eqref{H4} and \eqref{H5} are
  satisfied, as \eqref{H1} is included in the statement, \eqref{H2} is
  trivial here since $F=0$, and \eqref{H3} is a consequence of \eqref{H4}
  since the operator $N$ in equation \eqref{eq:spe} does not depend on
  time. We point out that due to Remark \ref{rem:in-omega} we may
  use Theorem \ref{thm:well-posedness} in $\mathcal{M}(\Omega)$, as we
  have no drift term here ($F = 0$).

  For \eqref{H4} we need to prove that given $R > 0$ there exists a
  constant $L > 0$ such that
  \begin{equation*}
    \|N(\mu) - N(\nu)\| \leq L \|\mu - \nu\|
  \end{equation*}
  for all $\mu, \nu \in \mathcal{M}$ with $\|\mu\|, \|\nu\| \leq
  R$. For the first term in \eqref{eq:spe},
  \begin{equation}
    \label{eq:t1}
    \|b \mu - b \nu\| = \|b(\nu-\mu)\|
    \leq \|b\|_{1,\infty} \|\nu - \mu\|.
  \end{equation}
  For the second,
  \begin{multline}
    \label{eq:t2}
    \| m(\cdot,P(\mu)) \mu - m(\cdot,P(\nu)) \nu\|
    \\
    \leq
    \| \big(m(\cdot,P(\mu)) - m(\cdot,P(\nu))\big) \mu\|
    +
    \| m(\cdot,P(\nu)) (\mu-\nu)\|
    \\
    \leq
    \| m(\cdot,P(\mu)) - m(\cdot,P(\nu)) \|_{1,\infty} \| \mu\|
    +
    \| m(\cdot,P(\nu))\|_{1,\infty} \|\mu-\nu\|
    \\
    \leq
    L_m | P(\mu) - P(\nu) | \| \mu\|
    +
    C \|\mu-\nu\|
    \leq
    \| \mu - \nu \| (C + L_m \| \mu\|),
  \end{multline}
  where $C$ is a constant such that $\| m(\cdot, p) \|_{1,\infty} \leq
  C$ for all $p \in [-R,R]$ finite due to \eqref{eq:hyp-m2}. Finally,
  in order to estimate the third term we notice that, for all $\psi
  \in \mathcal{L}$,
  \begin{equation}
    \label{eq:hyp-gamma-consequence}
    \Big\| \int \gamma(x,\cdot) \psi(x) \,dx \Big\|_{1,\infty}
    \leq
    C,
  \end{equation}
  for some $C > 0$. Indeed, $\int \gamma(x,y) \psi(x) \,dx$ is
  uniformly bounded for $y \in \Omega$ due to the fact that
  $\gamma(\cdot,y)$ is a probability measure, and it is also Lipschitz
  in $y$ since
  \begin{equation*}
    \left| \int (\gamma(x,y) - \gamma(x,z)) \psi(x) \,dx \right|
    \leq
    \|\gamma(\cdot,y) - \gamma(\cdot,z)\|
    \leq L_\gamma |y-z|
  \end{equation*}
  for all $y,z \in \Omega$, due to \eqref{eq:hyp-gamma-Lip}. Hence,
  \eqref{eq:hyp-gamma-consequence} holds and we can estimate the third
  term in~\eqref{eq:spe} by integrating against a function $\psi \in
  \mathcal{L}$:
  \begin{align}
    \label{eq:3rd-term}
    &\left| \int \int b(y) \gamma(x,y)
      (\mu(y) - \nu(y)) \psi(x) \,dy \,dx \right|
    \\
    &\qquad=
    \abs{ \int (\mu(y) - \nu(y)) b(y)
      \int \gamma(x,y) \psi(x) \,dx
      \,dy }
    \nonumber\\
    &\qquad\leq
    \norm{\mu - \nu}
    \norm{b}_{1,\infty}
    \norm{\int \gamma(x,\cdot) \psi(x) \,dx}_{1,\infty}
    \leq
    C \norm{\mu - \nu}
    \norm{b}_{1,\infty}.\nonumber
  \end{align}
Putting together \eqref{eq:t1}, \eqref{eq:t2}, and
\eqref{eq:3rd-term} we conclude that \eqref{H4} holds.
  Finally, \eqref{H5} is easily seen to hold using that $b$ is bounded
  and $\gamma(\cdot,y)$ is a probability measure.
\end{proof}

In the general abstract theorem we do not show conservation of
positivity for solutions. Since $L^1(\Omega)$ is dense in
$\mathcal{M}(\Omega)$ in the bounded Lipschitz distance, this is a
straightforward consequence of the result of conservation of
positivity of $L^1$ solutions, which is already available for all of
the models mentioned here. We show positivity of solutions for this
model for the sake of completeness, but this will be skipped for the
rest of the models of the paper, to which analogous arguments are
applicable.

\begin{lem}
  \label{lem:positivity}
  Under the hypotheses of Theorem \ref{thm:selection-mutation}, every
  solution of \eqref{eq:spe} with positive initial condition $u_0$ is
  positive.
\end{lem}

\begin{proof}
  We begin by showing positivity of local solutions of
  \eqref{eq:spe} in $L^{1}(\R^d)$. The initial value problem can be
  written as
  \begin{equation}
    \label{ivp}
    \begin{cases}
      &\frac{\partial u}{\partial t}=Au+f(u)
      \\
      &u(0)=u_{0}
    \end{cases}
  \end{equation}
  where
  $$
  Au(t,x) := (1-\varepsilon) b(x) u(t,x)
  +\varepsilon \int_{\Omega} b(y) \gamma(x,y) u(t,y) \,dy,
  $$
  and $f(u)(t,x):= - m(x,P(t)) u(t,x)$. The operator $A$ is the generator
  of a positive semigroup $T(t)$. Let $\lambda$ be a constant bigger
  than the bound of $m$. If we add and subtract $\lambda u$ to
  (\ref{ivp}) we get
  $$
  \frac{\partial u}{\partial t} = \big(A - \lambda I\big) u +
f(u(t)) + \lambda u(t),
$$
The mild solutions of this new initial value problem, and therefore
also those of problem (\ref{ivp}), can be constructed by iterating a
suitable variation of constants formula \cite{Pazy,Nagel}. More
precisely, they are limits of the sequence $(z_{n})_{n \geq 0}$ of
functions defined on $[0, t_{\max})$ for some $t_{\max} > 0$,
recursively defined by the formula
\begin{equation*}
  z_{n+1}(t)
  =
  \widetilde{T}(t)z_0 +
  \int_{0}^{t} \widetilde{T}(t-s) \big( f(z_n(s))
  + \lambda z_n(s)\big) \,ds,
\end{equation*}
where $\widetilde{T}(t)$ is the semigroup generated by the operator
$A-\lambda I$, that is $\widetilde{T}(t) = e^{-\lambda t}T(t)$.

Since $z_0$ is positive, the semigroup $\widetilde{T}(t)$ is positive
and since $\lambda$ is larger than the bound of $m$, we obtain that
$z_1$ is positive. By induction over $n$ we have that $(z_{n})_{n \geq
  0}$ is positive. Finally, since the cone of the positive functions
of $L^{1}$ is closed, we obtain that $z(t)$ is positive.  Positivity
of local solutions implies positivity of global solutions by a
standard connectedness argument. Finally, using Theorem
\ref{thm:well-posedness}, the density of $L^1(\Omega)$ in
$\mathcal{M}(\Omega)$ in the bounded Lipschitz distance gives us
conservation of positivity in the space of measures.
\end{proof}

\subsection{A selection-mutation model with infinite-dimensional
  environment}
Another example of a selection-mutation equation is
\begin{equation}
  \label{selmut}
  \begin{split}
    \frac{\partial u}{\partial t}(t,x)
    = &\,\Big((1-\varepsilon)b(x)-d_{0}(x)
    -\int_{\Omega}d(x,y)u(t,y)\,dy\Big) u(t,x)
    \\
    &+ \varepsilon \int_{\Omega}b(y) \gamma(x,y)u(t,y)\,dy
    =: N(u(t,\cdot))(x).
  \end{split}
\end{equation}
for the density of individuals $u(t,x)$ with respect to an (abstract)
evolutionary trait $x \in \Omega$. The difference with (\ref{eq:spe})
is that here the trait-specific per capita death rate is given by the
sum of the terms $d_{0}(x)$ and $\int_{\Omega}d(x,y)u(t,y)\,dy$. The
latter one models the interaction between individuals through
competition for resources, and is the only nonlinear term in the
equation (whose nonlinearity in this case is infinite dimensional).

This model was presented in \cite{ccdr}, where the authors prove
existence of steady states and also that their asymptotic profile when
the mutation rate $\varepsilon\to 0$ is a Cauchy distribution.
Our well-posedness result in the space of measures for equation
\eqref{selmut} is the following:

\begin{thm}
  \label{thm:smi}
  Assume \eqref{H1}, points (\ref{it:sm-b}) and (\ref{it:sm-gamma}) in
  Theorem {\rm\ref{thm:selection-mutation}}, and also that $d_0$ and
  $d$ are nonnegative functions satisfying
  \begin{equation}
    \label{eq:hyp-d0}
    d_0: \Omega \to \R \text{ is in } W^{1,\infty}(\Omega),
  \end{equation}
  and $d: \Omega \times \Omega \to \R$ with $d \in
  W^{1,\infty}(\Omega; W^{1,\infty}(\Omega))$; that is, there exists
  $L > 0$ such that
  \begin{gather}
    \label{eq:hyp-d-1}
    \norm{d(x,\cdot)}_{W^{1,\infty}(\Omega)} \leq L
    \quad \text{ for all } x \in \Omega.
    \\
    \label{eq:hyp-d-2}
    \norm{d(x, \cdot) - d(z,\cdot)}_{W^{1,\infty}(\Omega)}
    \leq
    L |x - z|
    \quad \text{ for all } x, z \in \Omega.
  \end{gather}
  Then the operator $N$ in equation \eqref{selmut} satisfies the
hypotheses
  of Theorem {\rm\ref{thm:well-posedness}}. Consequently, equation
  \eqref{selmut} is well-posed in the sense of Theorem
  {\rm\ref{thm:well-posedness}}.
\end{thm}

\begin{proof}
  As remarked in the proof of Theorem \ref{thm:selection-mutation}, we
  only need to check \eqref{H4} and \eqref{H5}. As the other terms have
  the same form as the terms in \eqref{eq:spe}
  since \eqref{eq:hyp-d0} is satisfied, we only need to check
  \eqref{H4} and \eqref{H5} for the term which involves $d$.

  First we notice that due to (\ref{eq:hyp-d-1})--(\ref{eq:hyp-d-2})
  the term $\int_{\Omega} d(x,y) u(y)\,dy$ is in
  $W^{1,\infty}(\Omega)$ for any $u \in \mathcal{M}(\Omega)$, as for
  all $x \in \Omega$,
  \begin{equation*}
    \left| \int_{\Omega} d(x,y) u(y)\,dy \right|
    \leq
    \|d(x,\cdot)\|_{W^{1,\infty}(\Omega)} \, \|u\|
    \leq
    L \|u\|.
  \end{equation*}
  and for any $x,z \in \Omega$,
  \begin{equation*}
    \left| \int_{\Omega} (d(x,y)-d(z,y)) u(y)\,dy \right|
    \leq
    \|d(x,\cdot) - d(z,\cdot)\|_{W^{1,\infty}(\Omega)} \, \|u\|
    \leq
    L |x-z| \, \|u\|.
  \end{equation*}
  Actually, we have proved that for any $w \in {\mathcal M}(\Omega)$,
  \begin{equation}
    \label{eq:d-prev}
    \norm{ \int_{\Omega} d(\cdot,y) w(y)\,dy }_{W^{1,\infty}(\Omega)}
    \leq
    L \|w\|.
  \end{equation}
  In order to prove \eqref{H4} for the term involving $d$,
  take two measures $u$, $v$ in $\mathcal{M}(\Omega)$. Then,
  \begin{align*}
    &\left\|
      u \int_{\Omega} d(\cdot,y) u(y)\,dy
      - v \int_{\Omega} d(\cdot,y) v(y)\,dy
    \right\|
    \\
    &\qquad\leq
    \left\|
      (u-v) \int_{\Omega} d(\cdot,y) u(y)\,dy
    \right\|
    + \left\|
      v \int_{\Omega} d(\cdot,y) (u(y)-v(y))\,dy
    \right\|
    \\
    &\qquad\leq
    \|u-v\|
    \left\|
      \int_{\Omega} d(\cdot,y) u(y)\,dy
    \right\|_{1,\infty}
    + \|v\|
    \left\|
      \int_{\Omega} d(\cdot,y) (u(y)-v(y))\,dy
    \right\|_{1,\infty}
    \\
    &\qquad\leq
    L \|u-v\| \|u\|
    + L \|v\| \|u-v\|,
  \end{align*}
  where we used (\ref{eq:d-prev}) for the last step. This proves
  \eqref{H4}. On the other hand, \eqref{H5} is easily proved since, in
  particular, $\left| \int_{\Omega} d(x,y) u(y)\,dy \right| \leq L
  \|u\|_{\mathrm{TV}}$ for all $x \in \Omega$.
\end{proof}

\subsection{A selection-mutation model with age structure}
\label{sec:age-structured}

Let us consider the following equation
\begin{subequations}
  \label{eq:spe-age}
  \begin{align}
    \frac{\partial u}{\partial t} (t,a,x)
    +\frac{\partial u}{\partial a} (t,a,x)
    =& - m(a,x,P(u),Q(u)) u(t,a,x)
    \\
    u(t,0,x) =& \,(1- \varepsilon)\int_{x}^{\infty}b(a,x)u(t,a,x)\,da
    \\
    & +\varepsilon \int_{0}^{\infty}\int_{y}^{\infty}
    \gamma(x,y)b(a,y)u(t,a,y) \,da \,dy
    \\
    u(0,a,x) =& \ u_{0}(a,x)
  \end{align}
\end{subequations}
where $u(t,a,x)$ is the density of individuals with age $a \geq 0$ and
maturation age $x \geq 0$ (the evolutionary variable) at time
$t$. $P(u)$ and $Q(u)$ denote, respectively, the total population of
juveniles and adults at time $t$, that is
$P(u)=\int_{0}^{\infty}\int_{0}^{x}u(t,a,x)\,da\,dx$,
$Q(u)=\int_{0}^{\infty}\int_{x}^{\infty}u(t,a,x)\,da\,dx$, $m$ is the
mortality rate, $b$ is the fertility rate and $\gamma(x,y)$ is the
probability density that the maturation age of the mutant offspring of
an individual with maturation age $y$ is $x$. As in the previous
examples, $\varepsilon$ stands for the probability of mutation.

This model is a slightly modified version of the one studied in
\cite{palmada}, where the only difference is in the term of inflow of
newborns. The difference of this model with the ones in the previous
sections is that here, fixing the evolutionary variable, we still have
an infinite-dimensional model, more precisely, an age-structured
population model. In \cite{palmada} well-posedness of the model was
proved in the Banach space $L^{1}(\mathbb{R}^{2}_+)$ (denoting
$\R^2_+=[0,+\infty) \times [0,+\infty)$), and also the existence of
steady states. In order to show well-posedness in the space of
measures we rewrite \eqref{eq:spe-age} as follows:
\begin{subequations}
  \label{eq:spe-age-2}
  \begin{gather}
    \partial_t u + \partial_a u =  N_1(u)
    \\
    u(t,0,x) = n_2(u)
    \\
    u(0,a,x) = u_{0}(a,x)
  \end{gather}
\end{subequations}
where we call $\Omega := \mathbb{R}^{2}_+$ and define, for $u \in
\mathcal{M}(\Omega)$,
\begin{align*}
%  \label{eq:def-N1}
  N_1(u) := &\,- m(a,x,P(u),Q(u)) u,
  \\
%  \label{eq:def-N2}
    n_2(u)(x) := &\,(1-\varepsilon) \int_{x}^{\infty}b(a,x) u(a,x) \,da
    \\& + \varepsilon \int_{0}^{\infty}\int_{y}^{\infty}\gamma(x,y)b(a,y)u(t,a,y)\,da\,dy.
\end{align*}
The model may be rewritten in the form (\ref{eq:abstract}) by
extending it to an equation on $\R^2$, with an additional independent
term. Let us be precise about the intended solutions:
\begin{dfn}
  \label{dfn:solution-age}
  Take $T \in [0,+\infty]$. We say a continuous function $u:[0,T) \to
  \mathcal{M}(\Omega)$ is a \emph{solution of equation
  \eqref{eq:spe-age-2} on $[0,T)$ with initial condition $u_0$}
  when, for every $\phi \in \mathcal{C}^\infty_0([0,T) \times
  \Omega)$,
  \begin{multline*}
%    \label{eq:def-solution-age}
      - \int_0^T \int_{\Omega} u \partial_t \phi
      \,dx\,da\,dt
      - \int_{\Omega} \phi(0,x,a) u_0(x,a) \,dx\,da
      \\
      - \int_0^T \int_{\Omega}
      u \partial_a \phi \,dx\,da \,dt
      - \int_0^T \int_0^\infty n_2(u(t)) \phi(t,x,0) \,dx \,dt
      \\
      =  \int_0^T \int_{\Omega} N_1(u(t)) \phi \,dx\,da\,dt.
    \end{multline*}
    (When the variables of $\phi$ or $u$ are not specified, it is
    understood that they are $(t,a,x)$).
\end{dfn}

We now take a suitable extension of the functions $m$, $b$ and
$\gamma$ to all of $\R^2$ (for definiteness, by mirror symmetry first
in $x$ and then in $a$) and consider the following equation, posed in
the whole set of $(a,x) \in \R^2$:
\begin{subequations}
  \label{eq:age-rewrite}
  \begin{align}
    &\frac{\partial u}{\partial t} + \frac{\partial u}{\partial a} =
    N_1(u) + n_2(u) \delta_{a=0},
    \\
    &u(0) = u_0.
  \end{align}
\end{subequations}
Equation \eqref{eq:age-rewrite} is of the form
\eqref{eq:abstract}. Now, observe that a solution of
\eqref{eq:age-rewrite}, in the sense of Definition
{\rm\ref{dfn:solution-abstract}}, is also a solution to
\eqref{eq:spe-age-2} in the sense of Definition
{\rm\ref{dfn:solution-age}} when restricted to $\mathbb{R}^{2}_+$,
provided it is zero on the set $\R^2 \setminus \R^2_+$. Hence, we just
need to give conditions on $m$, $b$ and $\gamma$ so that
\eqref{eq:age-rewrite} satisfies Hypotheses \eqref{H1}--\eqref{H5}
\emph{and} its solutions are supported on $\R^2_+$.

\begin{thm}
  \label{thm:age}
  We assume the following:
  \begin{enumerate}
  \item $b \in W^{1,\infty}(\Omega)$, and it is nonnegative.
  \item $m: \Omega \times \R \times \R \to \R$ is a nonnegative
    function satisfying a condition similar to the one in Theorem
    {\rm\ref{thm:selection-mutation}}: for each $p,q \in \R$, $m(\cdot,
    p,q) \in W^{1,\infty}$, and for each $R > 0$ there exists $L_m > 0$
    such that
    \begin{equation*}
%      \label{eq:hyp-m22}
      \norm{m(\cdot, p_1,q_1) - m(\cdot, p_2,q_2)}_{1,\infty}
      \leq
      L_m (|p_1 - p_2| + |q_1 - q_2|)
      \end{equation*}
      for all $p_1, p_2, q_1, q_2 \in [-R,R]$.
    \item
    For each $y \in \R$, $\gamma(\cdot, y)$ is a positive
    probability measure on $\R$ and
    there exists $L_\gamma > 0$ such that
    \begin{equation*}
%      \label{eq:hyp-gamma-Lip2}
      \|\gamma(\cdot,y) - \gamma(\cdot,z)\|
      \leq
      L_\gamma |y-z|
      \quad \text{ for all } y,z \in \R.
    \end{equation*}
  \end{enumerate}
  Then the initial value problems \eqref{eq:age-rewrite} and
  \eqref{eq:spe-age} are well-posed in the sense of Theorem
  {\rm\ref{thm:well-posedness}}.
\end{thm}

\begin{proof}
  As remarked in the proof of Theorem \ref{thm:selection-mutation}, we
  only need to prove \eqref{H4} and \eqref{H5}, and we can do it
  separately for each term.

  The term $N_1(u)$ can be treated in a similar way to the term
  $m(x,P) u$ in \eqref{eq:spe}, and we omit the details. For the
  term $n_2(u) \delta_{a=0}$ we have
  \begin{equation*}
    \norm{(n_2(u) - n_2(v)) \delta_{a=0}}_{\mathcal{M}(\R^2)}
    = \norm{n_2(u) - n_2(v)}_{\mathcal{M}(\R)},
  \end{equation*}
  where the norm on $\mathcal{M}(\R)$ and $\mathcal{M}(\R^2)$ is the
  bounded Lipschitz norm. The term in $n_2(u)$ which involves $\gamma$
  is of a similar form to the one in \eqref{eq:spe} and can be treated
  analogously.  For the other term, taking any test function $\phi \in
  W^{1,\infty}(\R)$ with $\|\phi\|_{1,\infty} \leq 1$,
  \begin{multline*}
    \int_\R \int_\R \phi(x) b(a,x) (u(a,x) - v(a,x)) \,da \,dx
    \\
    \leq
    \|u-v\| \, \|\phi\, b \|_{1,\infty}
    \leq
    \|u-v\| \, \| b \|_{1,\infty},
  \end{multline*}
  which shows that
  \begin{equation*}
    \norm{
      \int_\R b(a,\cdot) (u(a,\cdot) - v(a,\cdot)) \,da
    } \leq \|u-v\| \, \| b \|_{1,\infty},
  \end{equation*}
  hence proving \eqref{H4} for this term. Condition \eqref{H5} for this
  term is easily seen to hold by using that $b$ is bounded.

  The above allows us to apply Theorem \ref{thm:well-posedness} to
  equation (\ref{eq:age-rewrite}). We deduce that the problem
  (\ref{eq:age-rewrite}) is well-posed, and we only have to show that
  its solutions have support on $\R^2_+$, so that they are also
  solutions to (\ref{eq:spe-age}). In order to do this, we take any
  time $T > 0$ and any nonnegative test function
  $\phi_T \in {\mathcal C}^\infty(\R^2)$ with compact support on
  $\R^2 \setminus \R^2_+$ and consider
  $\phi:[0,T]\times \R \times \R \to [0,+\infty)$ to be the solution
  to
  \begin{equation}
    \label{eq:phi}
    \frac{\partial \phi}{\partial t}
    + \frac{\partial \phi}{\partial a}
    = 0
    \quad \text{ on } (0,T) \times \R \times \R,
  \end{equation}
  with $\phi(T,a,x) = \phi_T(a,x)$ for $a,x \in \R$. (That is,
  $\phi(t,a,x) = \phi_T(a + T-t,x)$ for $t \in (0,T]$, $a,x \in \R$.)
  Then, noticing that
  \begin{equation*}
    \partial_t |u| + \partial_a |u| = N_1(u) \operatorname{sign}(u) +
    n_2(u) \delta_{a=0} \operatorname{sign}(u)
  \end{equation*}
  we have
  \begin{multline}
    \label{eq:1}
    \frac{d}{dt} \int_\R \int_\R \phi |u| \,da \,dx
    =
    -
    \int_\R \int_\R |u| \partial_a \phi \,da \,dx
    \\
    +
    \int_\R \int_\R \phi \big( -\partial_a |u| + (N_1(u)
    + n_2(u) \delta_{a=0} ) \operatorname{sign}(u) \big) \,da
    \,dx
    \\
    =
    \int_\R \int_\R \phi N_1(u) \operatorname{sign}(u) \,da \,dx
    \leq
    C \int_\R \int_\R \phi |u| \,da \,dx
  \end{multline}
  for some $C > 0$. Here we have used that for all $t \in [0,T]$,
  $\phi(t,\cdot,\cdot)$ has support contained in
  $\R^2 \setminus \R^2_+$, so that $\phi\, \delta_{a=0} = 0$. For the
  last inequality we used that both $P(u)$ and $Q(u)$ are bounded on
  $[0,T]$ (since the solution $u$ in bounded in $\mathcal{M}(\R^d)$),
  giving
  \begin{equation*}
    N_1(u) \operatorname{sign}(u)
    = m(a,x,P(u),Q(u)) |u| \leq C |u|
  \end{equation*}
  for some number $C > 0$ (due to condition (ii) in Theorem
  \ref{thm:age}). From \eqref{eq:1} we deduce that
  \begin{equation*}
    \int_\R \int_\R \phi |u| \,da \,dx
    = 0 \quad \text{ for $t \in [0,T]$},
  \end{equation*}
  since it is $0$ at $t=0$. In particular,
  $$
  \int_\R \int_\R \phi_T(a,x) |u(T,a,x)| \,da \,dx = 0
  $$
  and since $\phi_T$ was arbitrary we deduce that at time $T$, $u$ has
  support contained in $\R^2_+$.
\end{proof}

\subsection{An age-size structured model}

Let us consider the following age-size structured model from
\cite{W}:
\begin{subequations}
  \label{eq:age-size}
  \begin{align}
    \frac{\partial u}{\partial t} + \,\frac{\partial
      u}{\partial a}+ \frac{\partial }{\partial x} (g(x)u)
    & =  - m(a,x,P(t)) u
    \\
    & \text{ for } a \in (0,a_{1}), x \in(x_{0},x_{1}), \ t>0
    \\
    u(t,0,x) &=
    \int_{0}^{a_{1}} \int_{x_{0}}^{x_{1}}
    \beta(a,\hat{x},x)u(t,a,\hat{x})
    \,d\hat{x}\,da
    \\
    &\text{ for } x \in(x_{0},x_{1}), \ t>0
    \\
    u(0,a,x) &= \, u_{0}(a,x)
    \quad \text{for } a \in (0,a_{1}), \ x \in (x_{0},x_{1}),
  \end{align}
\end{subequations}
where $u=u(t,a,x)$ denotes the density of individuals with age $a$,
with $0 \leq a \leq a_{1} \leq \infty$ and size $x$ with $0 \leq x_{0}
\leq x \leq x_{1} \leq \infty$. Size increases with time, in the same
way for all individuals of the population, and the growth rate is
given by the function $g(x)$ which is assumed not to depend on
environmental factors. Moreover it satisfies $g(x) \geq 0$ and
$g(x_{0})=0$. $m$ denotes the mortality rate and $\beta(a,\hat{x},x)$
denotes the average number of offspring of size $x$ produced per unit
of time by an individual of age $a$ and size $\hat{x}$ and
$P(t)=\int_{0}^{a_{1}}\int_{x_{0}}^{x_{1}}u(t,a,x)\,dx\,da$. Here, we
denote by $\Omega = (0,a_1) \times (x_0,x_1)$ the domain of definition
of the equation.

Many versions of the model \eqref{eq:age-size}, both linear and
nonlinear, have been studied, for instance in \cite{MD} and also in
\cite{TZ} where a more general nonlinear model containing an arbitrary
number of structured variables is considered. The usual space to study
these models is $L^{1}(\Omega)$. By using essentially the same
ingredients as in the previous subsection, one can prove the following
theorem that we state without proof.

\begin{thm}
  \label{thm:age-size}
  We assume the following:
  \begin{enumerate}
  \item $m: \Omega \times \R \to \R$ is a nonnegative
    function satisfying a condition similar to the one in Theorem
    {\rm\ref{thm:selection-mutation}}: for each $p \in \R$, $m(\cdot,
    p) \in W^{1,\infty}$, and for each $R > 0$ there exists $L_m > 0$
    such that
    \begin{equation*}
%      \label{eq:hyp-m22}
      \norm{m(\cdot, p_1) - m(\cdot, p_2)}_{1,\infty}
      \leq
      L_m |p_1 - p_2|
      \end{equation*}
      for all $p_1, p_2 \in [-R,R]$.
    \item $g\in W^{1,\infty}([x_0,x_1])$ with $g(0)=0$ and
    $g(x_1)>0$.
  \item The map $\beta: \Omega \to \mathcal{M}([x_0,x_1])$
    assigns $(a,\hat x) \mapsto \beta(a,\hat x,\cdot)$ and
    verifies that $W^{1,\infty}(\Omega,\mathcal{M}([x_0,x_1]))$, i.e,
    it is bounded and there exists $L_\beta > 0$ such that
    \begin{equation*}
      \|\beta(a_1,\hat x_1,\cdot) - \beta(a_2,\hat x_2,\cdot)\|
      \leq
      L_\beta (|a_1-a_2|+|\hat x_1 -\hat x_2|)
    \end{equation*}
    for all $(a_1,\hat x_1),(a_2,\hat x_2) \in \Omega$.
  \end{enumerate}
  Then the initial boundary value problem to \eqref{eq:age-size} is well-posed in the sense of
  Remark {\rm\ref{rem:support2}}.
\end{thm}

\begin{rem}\label{rem:support2}
  Let us mention that the extension outside the realistic domain
  $\Omega$ to $\R^2$ of the model ingredients $m$, $\beta$, and $g$
  while meeting the conditions in Theorem \ref{thm:well-posedness} may
  be done in many different ways. Once one has an extended equation in
  $\R^2$, Theorem \ref{thm:well-posedness} applies, and all solutions
  to the extended equations lead to the same solution once restricted
  to $\Omega$. This is due to the fact that the characteristics
  associated to the transport field $(1,g(x))$ for the age and size
  variables $(a,x)$ are not incoming at the boundaries: $a=a_1$,
  $x=x_0$, and $x=x_1$. A similar argument as in the proof of Theorem
  \ref{thm:age} shows that if the solutions for these extended systems
  are zero initially in the set of $a<0$, then they remain so for all
  times.
\end{rem}

% \bibliographystyle{plain}
% \bibliography{bibliography}

\end{document}